\documentclass[12pt]{amsart}
 \newtheorem{theorem}{Theorem}[section]
 \newtheorem{cor}[theorem]{Corollary}
 \newtheorem{thm}[theorem]{Theorem}
 \newtheorem{lem}[theorem]{Lemma}
 \newtheorem{prop}[theorem]{Proposition}
 \theoremstyle{definition}
 
 \theoremstyle{remark}
 \newtheorem{remark}[theorem]{Remark}

\newtheorem{notation}[theorem]{Notation}

\newtheorem{exam}[theorem]{Example}

\newtheorem*{question*}{Question}
 
 \numberwithin{equation}{section}

\usepackage{amsmath}
\usepackage{amsfonts}
\usepackage{amssymb}

\newcommand{\qf}[1]{\langle #1\rangle}
\newcommand{\pff}[1]{\langle\!\langle #1\rangle\!\rangle}
\newcommand{\qpf}[1]{\langle\!\langle #1]]}
\newcommand{\HH}{{\mathbb H}}

\newcommand{\BB}{{\mathcal B}}

\newcommand{\MM}{{\mathbb M}}

\newcommand{\implyr}{\Longrightarrow}
\newcommand{\implylr}{\Longleftrightarrow}

\renewcommand{\phi}{\varphi}

\DeclareMathOperator{\chr}{char}
\DeclareMathOperator{\rad}{Rad}
\DeclareMathOperator{\Aut}{Aut}
\DeclareMathOperator{\Gal}{Gal}
\DeclareMathOperator{\GL}{GL}
\DeclareMathOperator{\ani}{an}

\begin{document}

\title[Similarity of quadratic and symmetric bilinear forms]{Similarity
of quadratic and symmetric bilinear forms in characteristic $2$}

\author{Detlev W.~Hoffmann}\thanks{The
author is supported in part by DFG grant HO 4784/2-1
{\em Quadratic forms, quadrics, sums of squares and
Kato's cohomology in positive characteristic}.}
\address{Fakult\"at f\"ur Mathematik,
Technische Universit\"at Dortmund,
D-44221 Dortmund,
Germany}
\email{detlev.hoffmann@math.tu-dortmund.de}

\date{}

\begin{abstract}
We say that a field extension $L/F$ has the descent 
property for isometry (resp. similarity) of
quadratic or symmetric bilinear forms
if any two forms defined over $F$ that become
isometric (resp. similar) over $L$ are already
isometric (resp. similar) over $F$.
The famous Artin-Springer theorem states that anisotropic
quadratic or symmetric bilinear forms over a field stay
anisotropic over an odd degree field extension.  As a 
consequence, odd degree extensions have the descent property
for isometry
of quadratic as well as symmetric bilinear forms.
While this is well known for nonsingular quadratic forms,
it is perhaps less well known for arbitrary quadratic or
symmetric bilinear forms in characteristic $2$.  We 
provide a proof in this situation.  More generally, we
show that odd degree extensions also have the
descent property for similarity.  Moreover, for symmetric
bilinear forms in characteristic $2$, one even
has the descent property for isometry and for similarity 
for arbitrary separable algebraic extensions.  We also
show Scharlau's norm principle for arbitrary quadratic
or bilinear forms in characteristic $2$.
\end{abstract}


\maketitle

\section{Introduction}

Springer's celebrated theorem \cite{sp} from 1952 states that anisotropic
quadratic or bilinear forms stay anisotropic over
odd degree field extensions (all bilinear forms in this
paper are assumed to be symmetric).  This anisotropy behaviour
is true in any
characteristic, although in his note, Springer assumed the field
to be of characteristic not $2$.  Apparently, E.~Artin had already
communicated a proof of this result
to E.~Witt in 1937 (see \cite[p.~41]{w}).  As a consequence, one obtains
that quadratic (or bilinear) forms that become isometric
over an odd degree field extension are already isometric over
the base field.  This result is quite well known
for nonsingular quadratic forms (so for example
in characteristic not $2$ as already noted by Springer in his note),
but perhaps less so for 
arbitrary quadratic forms (singular or not) in characteristic $2$,
or for bilinear forms in characteristic $2$.

For arbitrary quadratic forms, we will provide a proof of
this result in Proposition \ref{qfoddiso} for the reader's
convenience.
For bilinear forms in characteristic $2$, the result
may be known to the reader when the forms are anisotropic,
where it amounts to the injectivity of the restriction
map $WF\to WL$ of the Witt rings of bilinear forms
for an odd degree extension $L/F$ (see \cite[Cor.~18.6]{ekm}).
However, in characteristic $2$, bilinear forms of the
same dimension that are Witt equivalent need not be isometric
(cf.~Proposition \ref{milnor}).  We will show
something stronger in Proposition \ref{sepiso},
namely, that bilinear forms in 
characteristic $2$ that become isometric over a \emph{separable}
algebraic extension are already isometric over the base field.

The question whether isometry descends from a field extension
to the base field has a natural generalization simply
by replacing ``isometry'' by ``similarity''.
If $\phi$ and $\psi$ are
quadratic forms over a field $F$ of characteristic not $2$
that become similar over an odd degree field extension
$L$ of $F$, then they are already similar over 
the base field $F$.  This can be shown using well
established transfer arguments applied to the Witt
ring as has been noted by Sivatski \cite[Prop.~1.1]{si}
and by Black and Qu\'eguiner-Mathieu \cite[p.~378]{bqm}

In characteristic not $2$, the arguments proposed by 
Sivatski resp. by Black and Qu\'eguiner-Mathieu rely
on the fact 
that, using the Scharlau transfer and its Frobenius 
reciprocity property, it essentially suffices
to work in the Witt rings of $L$ resp. $F$ and use the fact
that the isometry class of a quadratic form is uniquely
determind by its dimension and its Witt class.

This proof cannot simply be transferred to the case of 
characteristic $2$.  Several problems arise.  Firstly,
one has to distinguish between quadratic forms and
bilinear forms.  In the case of nonsingular quadratic forms, 
it turns out that these transfer arguments can essentially
also be used to get the same result in characteristic $2$
(we will also make use of this fact, see Corollary \ref{nondegs*}).
However, we also
want to consider singular quadratic forms, so 
Witt equivalence arguments alone won't suffice to get the
desired result since we don't have Witt cancellation
in general.  For bilinear forms, one may restrict
to the consideration of nonsingular forms, 
but we still have the problem
that the isometry class is in general not uniquely determined
by the Witt class and the dimension.  Rather, one also needs
to consider totally singular quadratic forms as follows from 
Milnor's criterion for isometry of bilinear forms
(Proposition \ref{milnor} below).

The analogue of the descent result for similarity of
arbitrary quadratic forms
(singular or not) in characteristic $2$ 
for odd degree extensions will be
shown in Theorem \ref{qfoddsim}.  For bilinear forms
in characteristic $2$, we again show a stronger result,
namely, that the similarity property descends from separable
extensions to the base field, see Theorem \ref{sepsim}.
A crucial ingredient needed to obtain this strengthening is
the perhaps somewhat surprising fact that in
characteristic $2$, the group of similarity factors of
a bilinear form over a field $F$ is given by
the multiplicative group of some field extension of $F^2$
inside $F$.  This will be shown in Theorem \ref{bilsim}.
An analogous result for totally singular quadratic forms 
has been known, but it seems to be new for bilinear forms
in characteristic $2$.

Recall that Scharlau's norm principle for quadratic forms $q$
over a field $F$ of characteristic not $2$ states that if
$L/F$ is a finite extension and if $\lambda\in L^*$ is a similarity
factor of $q$ over $L$, then the norm $N_{L/K}(\lambda)$
is a similarity factor of $q$ over $F$ 
(see, e.g., \cite[Ch.~VII, Th.~4.3]{lam}).  
In characteristic $2$,
there are analogous
versions for nonsingular quadratic forms and anisotropic
bilinear forms (see \cite[Theorems 20.14, 20.17]{ekm}.
Our methods allow us to show that,
in fact, it holds for arbitrary quadratic and bilinear forms.

In \S~2, we collect some of the basic definitions and
facts regarding bilinear and quadratic forms in
characteristic $2$. In \S~3, we study similarity factors
of quadratic and in particular bilinear forms in characteristic $2$.
The descent property of odd degree extensions for isometry of
quadratic forms and of separable extensions for isometry of
bilinear forms will be established in \S~4.  The descent
property of separable extensions (resp. odd degree extensions)
for similarity of bilinear forms (resp. quadratic forms)
will be treated in \S~5 (resp. \S~6). The full characteristic $2$
version of Scharlau's norm principle will be shown in \S~7.

\section{Basic definitions and facts}
We refer the reader to \cite{ekm} concerning
all concepts and facts that we mention in the sequel
or that we use without explicit reference.
Unless stated otherwise, all fields we consider will be
of characteristic $2$. By a bilinear form over a field $F$,
we will always mean a symmetric bilinear form $b:V\times V\to F$
defined on a finite-dimensional $F$-vector space $V$.
If $b,b'$ are bilinear forms, we have the standard notions
of isometry of $b$ and $b'$, denoted by $b\cong b'$, and
of  their orthgonal sum $b\perp b'$.  They are called
similar (over $F$) if there exists some $\lambda\in F^*$
such that $b\cong \lambda b'$. 
The radical of $b$ is defined by
$\rad(b)=\{ x\in V\,|\,b(x,y)=0\ \forall y\in V\}$.
$b$ is called nonsingular (or regular) if $\rad(b)=\{ 0\}$.
If $\dim V=n$ and $\BB=(e_i)_{1\leq i\leq n}$ is a
basis of $V$, then we call $G_{b,\BB}=\bigl(b(e_i,e_j)\bigr)\in M_n(K)$
the Gram matrix of $b$ with respect to $V$.
Nonsingularity is equivalent to $\det(G_{b,\BB})\neq 0$.
Each $b$ decomposes into an orthogonal sum of a 
nonsingular part and its radical, and since isometries
of bilinear forms map radicals bijectively onto radicals
and induce isometries on the nonsingular parts,
we may assume for all our purposes that all bilinear forms
we consider will be nonsingular unless stated otherwise.
$b$ is called diagonalizable if the basis $\BB$ can be
chosen so that $G_{b,\BB}$ is a diagonal matrix
(in which case $\BB$ is called an orthogonal basis),
and we write $b\cong\qf{a_1,\ldots,a_n}_b$ where the 
$a_i\in F$ are the diagonal entries of $G_{b,\BB}$.

The value set of $b$ is defined to be 
$D_F(b)=\{ b(x,x)\,|\,x\in V\setminus\{ 0\}\}$, and we
put $D_F^0(b)=D_F(b)\cup\{ 0\}$ and 
$D_F^*(b)=D_F(b)\cap F^*$.  $b$ is called
isotropic if $D_F^0(b)=D_F(b)$, anisotropic otherwise
(i.e., $D_F^*(b)=D_F(b)$).
The isotropic binary bilinear forms are exactly the
metabolic planes $\MM(a)$ that have a Gram matrix of type
$\begin{pmatrix} 0 & 1 \\ 1 & a\end{pmatrix}$ for some
$a\in F$.  A bilinear form is called metabolic if it
is isometric to an orthogonal sum of metabolic planes:
$$\MM(a_1,\ldots,a_m):=\MM(a_1)\perp\ldots\perp\MM(a_n).$$
$\HH_b=\MM(0)$ is a called a hyperbolic plane, and a hyperbolic
bilinear form is a form of type $\MM(0,\ldots,0)$.
$b$ is hyperbolic iff $D_F^0(b)=\{ 0\}$ iff $b$
is not diagonalizable.  A bilinear form $b$ decomposes
as $b\cong b_m\perp b_{\ani}$ with $b_m$ metabolic
and $b_{\ani}$ anisotropic.  The anisotropic part is uniquely
determined up to isometry, but the metabolic part is not
(cf.~Proposition \ref{milnor} below).
We call $b$ and $b'$ Witt equivalent if $b_{\ani}\cong b'_{\ani}$.
The Witt equivalence classes form the (bilinear) Witt ring
$WF$ of $F$ with addition induced by the orthogonal sum and
multiplication induced by the tensor product $b\otimes b'$
of bilinear forms.  If $K/F$ is a field extension,
then the natural scalar extension of $b$ to $V_K=V\otimes K$
is denoted by $b_K$.  This induces a ring homomorphism
$r^*_{K/F}:WF\to WK$ called restriction.

A quadratic form $q$ on the $F$-vector space $V$ is a map
$q:V\to F$ such that
$q(\lambda x)=\lambda^2q(x)$ for all $\lambda\in F,x\in V$,
and such that 
$$b_q(x,y)=q(x+y)-q(x)-q(y),\ x,y\in V$$
defines a bilinear form on $V$.  We have again the usual
notions of isometries and orthogonal sums of quadratic forms
(where orthogonality is defined with respect to the 
associated bilinear form).  The radical of $q$ is defined
as $\rad(q):=\rad(b_q)$, and $q$ is said to be
nonsingular (or regular) if $\rad(q)=0$, and totally singular if
$\rad(q)=V$.  $q$ is nonsingular iff $q$ is isometric
to an orthogonal sum of binary quadratic forms of type
$[a_1,a_2]=a_1x^2+xy+a_2y^2$, $a_1,a_2\in F$.
Totally singular quadratic forms are exactly
the forms isometric to forms of type 
$$\qf{a_1\ldots,a_m}=a_1x_1^2+\ldots +a_mx_m^2.$$
Again, we define the value sets 
$D_F(q)=\{ q(x)\,|\,x\in V\setminus\{ 0\}\}$,
$D_F^0(q)=D_F(q)\cup\{ 0\}$ and $D_F^*(q)=D_F(q)\cap F^*$, 
and  we call $q$ isotropic
if $D_F^0(q)=D_F(q)$, anisotropic otherwise (i.e., $D_F^*(q)=D_F(q)$).
Also, we call $q$ a zero form if $\dim(q)=0$ or if
$D_F(q)=\{ 0\}$ (in which case $q$ is totally singular).
Note that if $q\cong\qf{a_1,\ldots,a_m}$ is totally singular, then
$D_F^0(q)=\sum_{i=1}^ma_iF^2$ is a finite-dimensional
$F^2$-linear supspace of $F$ which essentially determines
the isometry class of $q$, see Proposition \ref{qfwittiso} below.

A binary nonsingular quadratic form is isotropic iff
it is a (quadratic) hyperbolic plane $\HH$, i.e., 
isometric to a form of type $[0,0]\cong[0,a]$, $a\in F$.
A hyperbolic quadratic form is a form that is
isometric to an orthogonal sum of hyperbolic planes. 
Each quadratic form decomposes into
$$q\cong q_h\perp q_r\perp q_s\perp q_0,$$
where $q_0\cong\qf{0,\ldots,0}$ is a totally singular
zero form, $q_h$ is hyperbolic, $q_r$ is
nonsingular and $q_s$ is totally singular with 
$q_r\perp q_s$ anisotropic.  This anisotropic
part in such a decomposition is uniquely determined up
to isometry and called the anisotropic part $q_{\ani}$ of
$q$.  $q_s\perp q_0$ is the restriction of $q$ to
its radical and thus also uniquely determined,
and the anisotropic part $q_s$ of the radical is also
uniquely determined up to isometry. Also, $q_h$ is
uniquely determined in such a decomposition, however,
$q_r$ is generally not uniquely determined.
Counterexamples can easily be constructed by using
the formula
$$[a,b]\perp\qf{c}\cong [a,b+c]\perp\qf{c}.$$
This shows in particular that the Witt cancellation rule
$$q_1\perp q_2\cong q_1\perp q_3\quad\implyr\quad q_2\cong q_3$$
fails in general if $q_1$ is singular.  However,
this Witt cancellation rule \emph{does}
hold if $q_1$ is nonsingular (see, e.g., \cite[Th.~8.4]{ekm}).

We say that two quadratic forms $q,q'$ are Witt equivalent
iff $q_{\ani}\cong q'_{\ani}$.  The Witt classes of
nonsingular quadratic forms correspond to the elements of
the Witt group $W_qF$ of quadratic forms with addition
induced by the orthogonal sum.  $W_qF$ becomes a $WF$-module
induced by the tensor product of a (nonsingular)
bilinear form with a nonsingular quadratic form.
In particular, $\qf{a}_b\otimes q\cong aq$ (scaling with
$a\in F^*$).  Again, if $K/F$ is a field extension,
then $q$ gives naturally  rise to a quadratic form $q_K=q\otimes K$
and we have again a restriction map $r^*_{K/F}:W_qF\to W_qK$.

We have the following result relating Witt equivalence to
isometry of quadratic forms.
\begin{prop}\label{qfwittiso}
Let $q$ and $q'$ be quadratic forms over $F$ of the same
dimension.
\begin{enumerate}
\item[(i)]  If $q$ and $q'$ are totally singular, then 
$$q\sim q'\quad\implylr\quad q\cong q'\quad\implylr
\quad D_F^0(q)=D_F^0(q').$$
\item[(ii)]  If $q\cong q_r\perp q_s$ and 
$q'\cong q'_r\perp q'_s$ with $q_r,q'_r$
nonsingular and $q_s,q'_s$ totally singular.
Then 
$$\begin{array}{rcl}
q\cong q' & \implylr & q\sim q'\ \mbox{and}\ 
\dim q_s=\dim q'_s\\
 & \implylr & q_s\cong q'_s\ \mbox{and}\ 
q'_r\perp q_r\perp q_s\sim q_s.
\end{array}$$
\end{enumerate}
\end{prop}

\begin{proof}
(i)  See \cite[p.~56]{ekm} or \cite[Prop.~8.1]{hl}
(or \cite[Prop.~2.6]{h-pf} for a more general statement in 
arbitrary positive characteristic).

(ii) If $q\cong q'$ then clearly $q\sim q'$ by the 
uniqueness (up to isometry) of the anisotropic part,
and also $\dim q_s=\dim\rad(q)=\dim\rad(q')=\dim q'_s$.

For the converse, we use the above
decomposition $q\cong Q_h\perp Q_r\perp Q_s\perp Q_0$
with $Q_h$ hyperbolic, $Q_0$ a zero form and $Q_r\perp Q_s$
the anisotropic part with $Q_r$ nonsingular and $Q_s$ totally
singular
(and analogously for $q'$).  Note that
$q\sim q'$ means $Q_r\perp Q_s\cong Q'_r\perp Q'_s$,
in particular, $\dim (Q_r\perp Q_s)=\dim (Q'_r\perp Q'_s)$
and $\dim Q_s=\dim Q'_s$ since the radicals of the anisotropic
parts must also have the same dimension.  Together with
$\dim q=\dim q'$ and 
$\dim (Q_s\perp Q_0)=\dim q_s=\dim q'_s=\dim (Q'_s\perp Q'_0)$
we get $\dim  Q_0=\dim Q'_0$ and $\dim Q_h=\dim Q'_h$ and thus,
$Q_0\cong Q'_0$ and $Q_h\cong Q'_h$.

Now if $q\cong q'$ then $q_s\cong q'_s$, since the latter
are just the restrictions of the forms to their respective radical
and because isometries restrict to isometries on the radicals.
Then the nonsingular parts also have the same dimension
$m=\dim q_r=\dim q'_r$.  We then have
$$\begin{array}{rcl}
q'_r\perp q_r\perp q_s & \cong & 
q'_r\perp q'_r\perp q'_s\\
 & \cong & (m\times\HH)\perp q'_s\\
 & \cong & (m\times\HH)\perp q_s,\end{array}$$
hence $q'_r\perp q_r\perp q_s\sim q_s\cong q'_s$.

Conversely, if
this holds,
then by dimension count we must have
$$q'_r\perp q_r\perp q_s\cong (m\times\HH)\perp q'_s
\cong q'_r\perp q'_r\perp q'_s$$
and Witt cancellation ($q'_r$ is nonsingular!) yields
$q\cong q'$.
\end{proof}

We also have to relate Witt equivalence to isometry in the case
of bilinear forms.  If $b$ is a bilinear form on the $F$-vector
space $V$, then the quadratic form $q_b:V\to F$ associated to
$b$ and defined by $q_b(x)=b(x,x)$ is a totally singular quadratic form.
If we decompose $b$ in its anisotropic part (which is diagonalizable)
and a metabolic part as above, say,
$$b\cong \MM(a_1,\ldots,a_m)\perp\qf{a_{m+1},\ldots,a_n}_b,$$
then we have
$$q_b\cong\qf{\underbrace{0,\ldots,0}_m,a_1,a_2,\ldots,a_n}.$$
The following criterion essentially goes back to Milnor.

\begin{prop}[Milnor \cite{m}]\label{milnor}
Let $\alpha$ and $\beta$ be bilinear forms over $F$ of the same
dimension.  Then $\alpha\cong\beta$ if and only if 
$\alpha_{\ani}\cong\beta_{\ani}$ and $q_\alpha\cong q_\beta$.
\end{prop}

Finally, recall that an $n$-fold bilinear Pfister form
is a form of type $\pff{a_1,\ldots,a_n}_b=
\qf{1,a_1}_b\otimes\ldots\otimes\qf{1,a_n}_b$ and
an $(n+1)$-fold quadratic Pfister form is a form of
type $\qpf{a_1,\ldots,a_n,a_{n+1}}=
\pff{a_1,\ldots,a_n}_b\otimes [1,a_{n+1}]$ ($a_1,\ldots,a_n\in F^*$,
$a_{n+1}\in F$). An $n$-fold quasi-Pfister form is a totally
singular quadratic form of type 
$\pff{a_1,\ldots,a_n}=\qf{1,a_1}\otimes\ldots\otimes\qf{1,a_n}$
(here, we allow $a_i=0$).
Such (quasi-)Pfister forms have many nice properties, for example,
they are round, i.e., the nonzero values they represent are exactly
their similarity factors, and a bilinear (quadratic) Pfister form
is isotropic iff it is metabolic (hyperbolic).

\section{Relative similarity factors and the group of similarity factors}

Let $\phi$ and $\psi$ be bilinear
(resp. quadratic) forms over $F$ of the same 
dimension.  We define the set of relative similarity factors of
$\phi$ and $\psi$ as
$$G_F(\phi,\psi)=\{ c\in F^*\,|\,\phi\cong c\psi\}.$$
$G_F(\psi)=G_F(\psi,\psi)$ is the group of similarity
factors of $\psi$.  Note that $F^{*2}\subseteq G_F(\psi)$ and that
therefore we have  $G_F(\phi,\psi)=G_F(\psi,\phi)$, and if
$G_F(\phi,\psi)$ is nonempty, say, $c\in G_F(\phi,\psi)$,
then $G_F(\phi,\psi)=cG_F(\psi)$ is just a coset of 
the subgroup $G_F(\psi)\leq F^*$.
For later purposes, we put $G_F^0(\psi)=G_F(\psi)\cup\{ 0\}$
and $G_F^0(\phi,\psi)=G_F(\phi,\psi)\cup\{ 0\}$.
Note that if $\dim\psi =0$, then by definition
$G_F(\psi)=F^*$ and $G_F^0(\psi)=F$.

If $\psi$ is a bilinear or quadratic Pfister form or
a (totally singular) quasi-Pfister form, then
$G_F(\psi)=D^*_F(\psi)$ by the roundness property of 
(quasi-)Pfister forms (see \cite[Cor.~6.2, Cor.~9.9, Cor.~10.3]{ekm},
or \cite[Prop.~8.5]{hl}, \cite[Prop.~4.6]{h-pf} for quasi-Pfister forms).
Furthermore, for a quasi-Pfister form $\psi\cong\pff{a_1,\ldots,a_m}$
we have $G_F^0(\psi)=F^2(a_1,\ldots,a_m)$, a field
extension of $F^2$  of degree $2^k\leq 2^m$ (with equality iff
$\psi$ is anisotropic).  Here, we allow $m=0$, in which case 
$\psi=\qf{1}$ and $G_F^0(\psi)=F^2$.

For similarity factors of arbitrary totally singular quadratic forms,
we have the following.

\begin{lem}\label{tssim}
Let $\tau$ be a totally singular quadratic form over $F$.
If $\tau$ is the zero form, then $G_F^0(\tau)=F$.
Otherwise, let $m$ be the largest nonnegative integer
such that there exists an anisotropic $m$-fold quasi-Pfister form
$\pi$ and an anisotropic totally singular quadratic form
$\sigma$ such that $\tau_{\ani}\cong\pi\otimes\sigma$.
Then $G_F^0(\tau)=G_F^0(\pi)$ is a field extension
of degree $2^m$ of $F^2$ inside $F$.

In particular, if $\tau'$ is another totally singular quadratic form
with $\dim\tau=\dim\tau'$, then $G_F^0(\tau,\tau')$
is additively closed.
\end{lem}

\begin{proof}  For the statement concerning $G_F^0(\tau)$, see
\cite[Prop.~6.4]{h-pf} or \cite[Remark 10.4]{ekm}.
As for the remaining statement, if $G_F(\tau,\tau')=\emptyset$,
the result is clear.  If $c\in G_F(\tau,\tau')$, then
the result follows because then $G_F^0(\tau,\tau')=cG_F^0(\tau)$
and $G_F^0(\tau)$ is a field and thus additively closed.
\end{proof}

Next, we will show that for any bilinear form $\beta$,
$G_F^0(\beta)$ is in fact also always a field.

\begin{theorem}\label{bilsim}
Let $\beta$ be a bilinear form over $F$.
\item[(i)] $G_F(\beta)=G_F(\beta_{\ani})\cap G_F(q_\beta)$.
\item[(ii)] $G_F^0(q_\beta)$, $G_F^0(\beta_{\ani})$ and therefore
$G_F^0(\beta)$ are subfields of $F$ containing $F^2$.
They are finite extensions of $F^2$ if $\beta$ is not hyperbolic.
\item[(iii)] If $\alpha$ is a bilinear form over $F$ with
$\dim\alpha=\dim\beta$, then $G_F^0(\alpha,\beta)$ is
additively closed.
\end{theorem}

\begin{proof}
We may assume $\dim\beta>0$ as otherwise, the above statements
are trivial (all $G_F^0$'s in the statements will be equal to $F$).

(i) If $c\in F^*$, then $(c\beta)_{\ani}=c\beta_{\ani}$ and
$q_{c\beta}=cq_\beta$.
Hence, by Milnor's criterion (Proposition \ref{milnor}), we have
$\beta\cong c\beta$ if and only if
$\beta_{\ani}\cong c\beta_{\ani}$ and $q_\beta\cong cq_\beta$
and the claim follows.

(ii)  Note that $\beta$ is hyperbolic if and only if $q_\beta$
is a zero form,
in which case $\dim\beta_{\ani}=0$ and therefore
$G_F^0(q_\beta)=G_F^0(\beta_{\ani})=G_F^0(\beta)=F$.

If $\beta$ is metabolic but not hyperbolic, then we still have 
$\dim\beta_{\ani}=0$ but now $q_\beta$ is not
a zero form and we have 
$G_F^0(\beta)=G_F^0(q_\beta)$.

So suppose $\dim\beta_{\ani}>0$.  In view of part (i) and
Lemma \ref{tssim}, it suffices to assume that $\beta$ is
anisotropic and to show that $G_F^0(\beta)$ is a field containing $F^2$.
Clearly, $F^2\subseteq G_F^0(\beta)$.  It remains to show that
for any $a,b\in G_F^0(\beta)$ we have that $a-b=a+b\in G_F^0(\beta)$.
If $a=0$ this is clear.  So let us assume $a\neq 0$,
i.e., $a\in G_F(\beta)$ and thus $c:=ba^{-1}\in G_F^0(\beta)$.  
If $c\in F^2$, say,
$c=x^2$ with $x\in F$, then $a+b=b(1+x)^2\in G_F^0(\beta)$.
So let us assume $c\not\in F^2$ so that $\pff{c}_b$ is anisotropic. 

Now we have  $\beta\cong c\beta$,
and since $\beta$ is anisotropic, this is equivalent to
$\beta\perp c\beta$ being metabolic, i.e.,
$\pff{c}_b\otimes\beta =0$ in $WF$ (see, e.g.,
\cite[Prop.~2.4]{ekm}).  By \cite[Th.~4.4]{h-bil} (see also
\cite[Remark 2.2]{ab} or \cite[Cor.~6.23]{ekm})
and since $\pff{c}_b$ is anisotropic, we have
\begin{equation}\label{c0}
\pff{c}_b\otimes\beta =0\in WF\quad\implylr\quad
\beta\in \sum_{d\in D_F(\pff{c}_b)}WF\qf{1,d}_b.
\end{equation}
Now 
$$D_F(\pff{c}_b)=D_F(\qf{1,c})=D_F(\qf{1,c+1})=D_F(\pff{c+1}_b)$$
and therefore
\begin{equation}\label{c1}
\sum_{d\in D_F(\pff{c}_b)}WF\qf{1,d}_b=
\sum_{d\in D_F(\pff{c+1}_b)}WF\qf{1,d}_b.
\end{equation}
Eq.~\ref{c1} implies that in Eq.~\ref{c0}, we may replace $c$ by $c+1$
and conclude that
$\beta\cong (c+1)\beta$, i.e., $c+1\in G_F(\beta)$ and 
hence, $a+b=a(c+1)\in G_F(\beta)$.

(iii)  The argument is the same as in the proof of Lemma \ref{tssim}.
\end{proof}

\begin{exam}
(i) If $\beta$ is a bilinear form over $F$ with associated totally 
singular quadratic form $q_\beta$, then obviously
$G_F(\beta)\subseteq G_F(q_\beta)$.
In general, this inclusion is strict even for anisotropic $\beta$.
As an example, 
let $a,b\in F$ be $2$-independent elements, i.e., $[F^2(a,b):F^2]=4$.
Consider the bilinear form $\beta\cong\qf{a+1,a,b,ab}_b$.
Note that as totally singular quadratic forms, we have 
$$q_\beta\cong\qf{a+1,a,b,ab}\cong\qf{1,a,b,ab}\cong\pff{a,b},$$
which is a $2$-fold quasi-Pfister form, hence
$$G_F(q_\beta)=G_F(\pff{a,b})=D_F(\pff{a,b})=F^2(a,b)^*$$
(see, e.g., \cite[Prop.~8.5]{hl} or \cite[Cor.~10.3]{ekm}).
Note also that $q_\beta$ and hence also
$\beta$ are anisotropic.

However, as bilinear forms, comparing determinants shows that
that $\beta\not\cong\pff{a,b}_b$.  In fact, in $WF$ we have
$\beta = \pff{a+1}_b+\pff{a,b}_b$.  Since $\beta$ is anisotropic,
$x\in F^*$ is a similarity factor iff $\pff{x}_b\otimes\beta=0\in WF$
iff $\pff{x,a+1}_b=\pff{x,a,b}_b\in WF$, and since bilinear
Pfister forms are either anisotropic or metabolic, this holds iff
$\pff{x,a+1}_b=\pff{x,a,b}_b=0\in WF$ iff
$$x\in G_F(\pff{a+1}_b)\cap G_F(\pff{a,b}_b)=
F^2(a)^*\cap F^2(a,b)^*=F^2(a)^*.$$
Hence, $G_F(\beta)=F^2(a)^*\subsetneq F^2(a,b)^*=G_F(q_\beta)$.

(ii) It is possible that $G_F(q_\beta)\subsetneq G_F(\beta_{\ani})$.
But for this to hold, $\beta$ must be isotropic.
As an example, let $a,b$ be as in (i)
and consider $\beta\cong\pff{a}_b\perp\MM(b)$.  Then
$\beta_{\ani}\cong\pff{a}_b$ and hence $G_F(\beta_{\ani})=F^2(a)^*$.
On the other hand, $q_\beta\cong\qf{1,a,b,0}$  and 
thus $(q_\beta)_{\ani}\cong\qf{1,a,b}$, which is of odd
dimension.  Hence, by Lemma \ref{tssim},
$G_F(q_\beta)=F^{*2}\subsetneq F^2(a)^*=G_F(\beta_{\ani})$.  
In particular, $G_F^0(\beta)=F^2$.
\end{exam}

\section{Isometry of forms over certain algebraic extensions}
The main results of this section concern the descent of isometry
for totally
singular quadratic forms and for bilinear forms under separable
algebraic field extensions and for arbitrary quadratic forms under
odd degree extensions. 
We start with some general results on values represented by
totally singular quadratic forms under field extensions.

\begin{prop}\label{tsv}
Let $q$ be a nonzero totally singular quadratic form over $F$,
let $L/F$ be a finite field extension, let $a\in F$ and
$\lambda\in L\setminus F$, and let $K=F(\lambda)$.
\begin{enumerate}
\item[(i)]  If $L/F$ is separable and $a\in D_L(q_L)$, then
$a\in D_F(q)$.
\item[(ii)]  If $a\in F^*$ and $a\lambda\in D_K(q_K)$, then
$K/F$ is separable.
\item[(iii)]  If $a\in F^*$ and $a\lambda\in D_L(q_L)$, then
$K/F$ is separable or $L/K$ is inseparable.
\end{enumerate} 
\end{prop}

\begin{proof}  
The minimal polynomial of $\lambda$ over $F$ will be denoted
by $f(X)\in F[X]$, so $m=\deg(f)=[K:F]\geq 2$.
We write $q\cong\qf{a_1,\ldots,a_n}$, $a_i\in F$, not all $a_i=0$.

(i) Since $L/F$ is separable, it is a simple extension,
so without loss of generality, we may assume $L=K=F(\lambda)$.
Now $f(X)$ is a separable polynomial, so we have $f'\neq 0$
(where $f'$ denotes the formal derivative) and since
$\chr(F)=2$, we must have $f\in F[X]\setminus F[X^2]$.
Now $a\in D_L(q_L)$ is equivalent to the existence of $g_i\in F[X]$,
$1\leq i\leq n$, with $\deg(g_i)< m$, not all $g_i=0$, and
some $h\in F[X]$, such that
\begin{equation}\label{tsve1}
a=\sum_{i=1}^na_ig_i(X)^2 +h(X)f(X).
\end{equation}
Such an equation holds with the $g_i$ and $h$ not being divisible by
a common irreducible polynomial.  If $a\neq 0$, this is clear.
If $a=0$ and if $p(X)$ is a common irreducible divisor and $g_j\neq 0$, 
then $\deg(p)\leq\deg(g_j)<\deg(f)$, so $p$ does not divide
$f$.  On the other hand, $p^2|g_i^2$ (for each $i$)
and thus $p^2|h$ and we may
replace $g_i$ by $\frac{g_i}{p}$ and $h$ by $\frac{h}{p^2}$,
and in this way we can get rid of any common divisors. 

Comparing degrees on both sides in Eq.~\ref{tsve1} and using that
$\deg(g_i^2)\leq 2\deg(f)-2$, we see that
$\deg(h)<\deg(f)$.  Also, note that $g_i^2\in F[X^2]$, and comparing
monomials of odd degree on both sides and using the fact that
$f\in F[X]\setminus F[X^2]$, we must have
$h\in F[X]\setminus F[X^2]$ or $h=0$.  Consequently,
$h(X)$ must possess
an irreducible divisor $c(X)\in F[X]\setminus F[X^2]$
(if $h=0$, we may choose $c(X)=X$) and we have 
$\deg(c)<\deg(f)$.
In particular, $c(X)$ is a separable polynomial.
Let $\mu$ be a root of $c(X)$ in some algebraic closure of $F$.
Then $M=F(\mu)$ is separable over $F$ with
$$[M:F]=\deg(c)<\deg(f)=[L:F].$$
By what was mentioned above, not all the $g_i$ are divisible
by $c$, so we have $g_i(\mu)\neq 0$ for at least one
$i\in\{ 1,\ldots ,n\}$, and substituting in Eq.~\ref{tsve1}
yields $a=\sum_{i=1}^na_ig_i(\mu)^2$,
which shows that $a\in D_M(q_M)$ for a separable extension $M/F$
with $[M:F]<[L:F]$.  An induction on the degree of the
extension concludes the proof.

(ii)  Assume that $K/F$ is inseparable.  Then we must have
$f(X)\in F[X^2]$ and $\deg(f)=[K:F]\geq 2$
(recall that $\lambda\not\in F$).  Proceeding
as in (i), we now get an equation
\begin{equation}\label{tsve2}
aX=\sum_{i=1}^na_ig_i(X)^2 +h(X)f(X).
\end{equation}
Since $g_i(X)^2, f(X)\in F[X^2]$ and comparing monomials of
odd degree on both sides, we see that $h(X)\in F[X]\setminus F[X^2]$.
But then, $h(X)f(X)$ and therefore the right hand side of
Eq.~\ref{tsve2} contains a monomial of odd degree $\geq 3$
with nonzero coefficient, and comparing with the left hand side
yields a contradiction.

(iii) If $K/F$ is inseparable and $L/K$ is separable, then
(ii) implies that $a\lambda\not\in D_K(q_K)$, and (i) then
implies that $a\lambda\not\in D_L(q_L)$.
\end{proof}

This now yields a  generalization of
the Artin-Springer theorem on odd degree extensions.  Note that in
characteristic $2$, all odd degree field extensions are separable.

\begin{cor}\label{springersep}
Let $\phi$ be an anisotropic totally singular quadratic form
or a bilinear form over $F$, and let $L/F$ be a separable
algebraic extension. Then $\phi$ stays anisotropic over $L$.
\end{cor}

\begin{proof}
For totally singular forms, this follows readily
from Proposition \ref{tsv}(i) for $a=0$.
For a bilinear form $\phi$, this follows from the fact that the
isotropy of $\phi$ is nothing else but the isotropy of
the totally singular quadratic form $q_\phi$.
\end{proof}

\begin{remark}
For totally singular forms, the above corollary,
using different arguments, 
was originally shown in \cite[Prop~ 5.3]{h-pf}
and also in \cite[Lemma 2.8]{lag}).
Our proof (essentially the proof of Proposition \ref{tsv}(i))
mimics to some extent Springer's original
proof for odd degree extensions.
\end{remark}

\begin{prop}\label{sepiso}
Let $\phi$ and $\psi$ be totally singular quadratic forms or 
bilinear forms over $F$ of the same dimension.  Let $L/F$ be a 
separable algebraic field extension.  Then $\phi$ and $\psi$
are isometric over $F$ if and only if  $\phi$ and $\psi$ are
isometric over $L$.
\end{prop}

\begin{proof}
Of course, isometry over $F$ implies isometry over $L$.
As for the converse,
let us first consider the case where $\phi$ and $\psi$ are
totally singular quadratic forms, say, 
$\phi\cong\qf{a_1,\ldots,a_n}$,
$\psi\cong\qf{b_1,\ldots,b_n}$.  Now the isometry class
of $\phi$ over any field extension $K$ of $F$ is uniquely
determined by the $K^2$-vector space $D^0_K(\phi)$, which
is generated over $K^2$ by $a_1,\ldots,a_n$ (analogous
statements hold for $\psi$).  In other words, $\phi_K\cong\psi_K$
iff $D^0_K(\phi_K)=D^0_K(\psi_K)$ iff $a_i\in D^0_K(\psi_K)$ and
$b_i\in D^0_K(\phi_K)$ for all $1\leq i\leq n$. 
This and Proposition \ref{tsv}(i) together
readily yield that $\phi_L\cong\psi_L$ implies $\phi\cong\psi$.

Now suppose that $\phi$ and $\psi$ are anisotropic bilinear
forms of the same dimension over $F$.  Then $\phi\cong\psi$
iff $(\phi\perp\psi)_{\ani}=0$ (see, e.g., \cite[Prop.~2.4]{ekm}).
But because of Corollary \ref{springersep},
we have that $\phi_L$ and $\psi_L$  are anisotropic and
$((\phi\perp\psi)_{\ani})_L\cong (\phi_L\perp\psi_L)_{\ani}$, and therefore
$\phi\cong\psi$ iff $\phi_L\cong\psi_L$.

Finally, let $\phi$ and $\psi$ be arbitrary bilinear forms over
$F$ of the same dimension.  Then the result follows from the above
cases of totally singular quadratic forms and of anisotropic
bilinear forms together with Milnor's criterion (Lemma \ref{milnor}).
\end{proof}

For arbitrary quadratic forms and odd degree extensions, we 
have the following.
\begin{prop}\label{qfoddiso}
Let $\phi$ and $\psi$ be quadratic forms over $F$ and let
$L/F$ be an odd degree field extension.  Then $\phi$ and $\psi$
are isometric 
over $F$ if and only if $\phi$ and $\psi$ are isometric
over $L$
\end{prop}
\begin{proof} 
Write $\phi\cong\phi_r\perp\phi_s$ with $\phi_r$ nonsingular
and $\phi_s$ totally singular, and analogously $\psi\cong\psi_r\perp\psi_s$.
Recall that $\phi_s$ (resp. $\psi_s$) is uniquely determined
by $\phi$ (resp. $\psi$) as it corresponds to the radical of the form. 
The dimensions of the nonsingular resp. totally singular part do not
change under field extensions and remain unchanged under isometries.
Hence, we may in fact assume
$m=\dim\phi_r=\dim\psi_r$ and $\dim\phi_s=\dim\psi_s$.
By the Artin-Springer theorem and since the anisotropic parts of
quadratic forms are uniquely determined up to isometry, we may
furthermore assume that $\phi$ and $\psi$ are anisotropic.  Isometries
of quadratic forms restrict to isometries of their respective
radicals, hence $\phi_L\cong\psi_L$ implies 
$$(\phi_s)_L\cong (\phi_L)_s\cong (\psi_L)_s\cong (\psi_s)_L$$
and therefore, by Proposition \ref{sepiso}, 
$\phi_s\cong\psi_s$.  Also,
$$(\phi_r)_L\perp(\phi_s)_L\cong \phi_L\cong\psi_L\cong
(\psi_r)_L\perp (\phi_s)_L$$
and hence
$$(\psi_r\perp\phi_r\perp\phi_s)_L\cong (\psi_r\perp\psi_r\perp\phi_s)_L
\cong (m\HH\perp\phi_s)_L.$$
By the Artin-Springer theorem and since $\phi_s$ is anisotropic, this implies
$$\psi_r\perp\phi_r\perp\phi_s\cong m\HH\perp\phi_s.$$
Adding $\psi_r$ on both sides, and comparing anisotropic parts
(i.e. cancelling the hyperbolic planes on both sides)
yields 
$$\phi\cong \phi_r\perp\phi_s\cong\psi_r\perp\phi_s
\cong\psi_r\perp\psi_s\cong\psi.$$
\end{proof}

\section{Similarity of bilinear forms and totally
singular quadratic forms over separable extensions}

Before we turn our attention to similarity of bilinear forms resp.
totally singular quadratic forms over separable extensions,
let us introduce some additional terminology.  Let $K$ be a field,
$\sigma\in \Aut(K)$.  For any matrix $A=(a_{ij})\in M_n(K)$ we define
$\sigma(A)=\bigl(\sigma(a_{ij})\bigr)\in M_n(K)$.  Thus, $\sigma$ extends
to a ring automorphism of $M_n(K)$ and it respects congruence
of matrices, i.e., if $A,B\in M_n(K)$ are congruent iff $\sigma(A)$ and
$\sigma(B)$ are congruent.  Indeed, if there exists an $S\in\GL_n(K)$
with $S^tAS=B$, then by putting $T=\sigma(S)\in\GL_n(K)$, we have
$T^t\sigma(A)T=\sigma(B)$.  Obviously, the converse also holds
by putting $S=\sigma^{-1}(T)$.

Now let $\beta$ be any bilinear form on a $K$-vector space $V$,
and suppose that $\beta$ has Gram matrix $B\in M_n(K)$ for
some basis of $V$.  Then we define $\sigma(\beta)$ to be
the bilinear form with Gram matrix $\sigma(B)$.  By the above,
the isometry class of $\sigma(\beta)$ does not depend on the
chosen Gram matrix.  In particular, if $\alpha$ and $\beta$
are bilinear forms over $K$, then $\alpha\cong\beta$
if and only if $\sigma(\alpha)\cong\sigma(\beta)$.

If $K/F$ is a field extension and $\sigma\in\Gal(K/F)$,
and if $\beta$ is a bilinear form over $F$, then
obviously $\beta_K\cong\sigma(\beta_K)$ because
$\beta_K$ has a Gram matrix defined over $F$.

We get an analogous construction for totally singular
quadratic forms over $K$ by mapping such a form
$\phi\cong\qf{a_1,\ldots,a_n}$ to 
$\sigma(\phi)\cong\qf{\sigma(a_1),\ldots,\sigma(a_n)}$.
Again, if $\psi$ is another
totally singular quadratic form over $K$, then
$\phi\cong\psi$ iff  $\sigma(\phi)\cong\sigma(\psi)$.
This follows by comparing dimensions and from the fact
that $\sigma(K^2)=K^2$
and thus $D_K(\sigma(\phi))=\sigma(D_K(\phi))$, and
this equals $D_K(\sigma(\psi))=\sigma(D_K(\psi))$
if and only if $D_K(\phi)=D_K(\psi)$.

\begin{lem}\label{lambda}
Let $L/F$ be a separable
algebraic field extension and $\lambda\in L^*$.
Let $n=[F(\lambda):F]$, and let
$$f(X)=X^n+a_1X^{n-1}+\ldots a_{n-1}X+a_n=
\prod_{i=1}^n(X-\lambda_i)\in F[X]$$
be the minimal polynomial of $\lambda$ over $F$, where
$\lambda=\lambda_1,\lambda_2,\ldots,\lambda_n$ are the distinct
roots of $f$ in some algebraic closure containing $L$.
Let $\phi$ and $\psi$ be totally singular quadratic forms
or bilinear forms over $F$.  If $\phi_L\cong \lambda\psi_L$,
then for each $m\in\{ 1,\ldots , n\}$ the following holds: 
$$a_m\in\left\{
\begin{array}{ll}
G_F^0(\psi) & \mbox{if $m\equiv 0\bmod 2$,}\\
G_F^0(\phi,\psi) & \mbox{if $m\equiv 1\bmod 2$.}\end{array}
\right.$$
\end{lem}

\begin{proof}
We have
$$a_m=\sum_{1\leq i_1<i_2<\ldots <i_m\leq n}
\lambda_{i_1}\lambda_{i_2}\cdots\lambda_{i_m}$$
(we may disregard signs as we are in characteristic $2$).
Now consider the splitting field $K=F(\lambda_1,\ldots,\lambda_n)$
of $f$.  Note that the extensions $L/F(\lambda)$ and $K/F(\lambda)$
are both separable.  Hence, by Proposition \ref{sepiso},
we have 
$$\phi_L\cong\lambda\psi_L\quad \implylr\quad 
\phi_{F(\lambda)}\cong\lambda\psi_{F(\lambda)}\quad \implylr\quad 
\phi_K\cong\lambda\psi_K.$$
By Galois theory, there exists to each $i\in\{ 1,\ldots ,n\}$ some
$\sigma_i\in \Gal(K/F)$ with $\sigma_i(\lambda)=\lambda_i$.
By the construction described above and using that $\phi$, $\psi$
are defined  over $F$, we conclude that
$$\phi_K\cong\sigma_i(\phi_K)\cong\sigma_i(\lambda\psi_K)\cong
\sigma_i(\lambda)\psi_K\cong \lambda_i\psi_K.$$
This implies that $\lambda_i\in G_K(\phi_K,\psi_K)=\lambda G_K(\psi_K)$.
Hence, if $c_1,\ldots, c_m\in\{ \lambda_1,\ldots,\lambda_n\}$, then
$c=\prod_{i=1}^mc_i\in G_K(\psi_K)$ if $m$ is even and 
$c\in G_K(\phi_K,\psi_K)$ if $m$ is odd.

Since $G_K^0(\psi_K)$ and $G_K^0(\phi_K,\psi_K)$ are additively closed,
we readily get that $a_m\in G_K^0(\psi_K)$ if $m$ is even, and 
$a_m\in G_K^0(\phi_K,\psi_K)$ if $m$ is odd.  If $a_m=0$, the lemma is
trivially true.  If $a_m\neq 0$, we have 
$\psi_K\cong a_m\psi_K$ if $m$ is even and $\phi_K\cong a_m\psi_K$
if $m$ is odd, and since $K/F$ is
separable, we apply once  more Proposition \ref{sepiso} to conclude that,
already over $F$, we have
$\psi\cong a_m\psi$ if $m$ is even and $\phi\cong a_m\psi$ if $m$ is odd.
\end{proof}

\begin{thm}\label{sepsim}
Let $\phi$ and $\psi$ be totally singular quadratic forms
or bilinear forms over $F$.  Let $L/F$ be a separable
algebraic field extension.  Then $\phi$ and $\psi$ are
similar over $F$ if and only if $\phi$ and $\psi$ are similar over $L$.
\end{thm}

\begin{proof}  Similarity over $F$ trivially implies
similarity over $L$.  Conversely, let $\lambda\in L^*$ be such that
$\phi_L\cong\lambda\psi_L$.  We use the notations from 
Lemma \ref{lambda} applied to this $\lambda$,
Since $L/F$ is separable, $f\in F[X]$ is a separable
polynomial and hence, there is at least
one $m\in\{ 1,\ldots, n\}$ with $m\equiv 1\bmod 2$ and $a_m\neq 0$.
Indeed, if $n$ is odd, we may choose $m=n$, and if $n$ is even,
there must be such an $m$ because otherwise, we would get
$f'=0$ for the formal derivative of $f$, and thus $f$ would
not be separable.   By Lemma \ref{lambda}, we have
$\phi\cong a_m\psi$ over $F$.
\end{proof}

The case of totally singular quadratic forms can be
generalized to so-called $p$-forms.
Let $F$ be a field with $\chr(F)=p>0$.  Recall that a $p$-form
$\phi$ on a finite-dimensional $F$-vector space $V$ is a map
$\phi: V\to F$ satisfying $\phi(x+y)=\phi(x)+\phi(y)$ and 
$\phi(\lambda x)=\lambda^p\phi(x)$ for all $x,y\in V$ and 
all $\lambda\in F$.  For $p=2$, $p$-forms are just
totally singular quadratic forms and one can develop a uniform
theory of such forms for all $p$, including the notions
of isometry, similarity, $G_F(\phi)$, etc..  In particular,
each basis of the underlying vector space of a $p$-form
gives rise to a diagonalization $\phi\cong\qf{a_1,\ldots,a_n}$,
where the latter stands for the form
$a_1x_1^p+\ldots+a_nx_n^p$, $a_i\in F$.  A $1$-fold quasi-Pfister
form is a $p$-form of type $\pff{a}=\qf{1,a,a^2,\ldots,a^{p-1}}$,
and an $n$-fold quasi-Pfister form $\pff{a_1,\ldots,a_n}$ is the
product $\bigotimes_{i=1}^n\pff{a_i}$.  
Those basic properties of $p$-forms
which we need are  the same as for
totally singular quadratic forms and can be proved in essentially
the same way just by replacing $2$ by $p$ in the relevant places,
we refer to \cite{h-pf} for more details.   In particular,
Lemma \ref{tssim}, Corollary \ref{springersep}
and Proposition \ref{sepiso} also hold for $p$-forms.
We are now in a position to adapt the proof of
Theorem \ref{sepsim} (resp. Lemma \ref{lambda}) to $p$-forms.  We state it
as a corollary since its proof is just a slight modification
of the proofs above.

\begin{cor}  Let $F$ be a field of characteristic $p>0$
and let $\phi$ and $\psi$ be $p$-forms of the same dimension
over $F$.  Let $L/F$ be a separable
algebraic field extension.  Then $\phi$ and $\psi$ are
similar over $F$ if and only if $\phi$ and $\psi$ are similar over $L$.
\end{cor}

\begin{proof}  The ``only if''-part is trivial.  For the converse,
let $\lambda\in L^*$ be such that $\phi_L\cong\lambda\psi_L$.
We follow the proof of Lemma \ref{lambda} resp.\ Theorem \ref{sepsim} with the same
notations as there ($f(X)$, $\lambda_i$, $a_i$, etc.).
Firstly, we may assume that
$L=K$ is in fact the splitting field of the minimal polynomial
$f(X)\in F[X]$ of $\lambda$.
We still have that $G_L(\phi_L,\psi_L)=\lambda G_L(\psi_L)$ is a coset
of $G_L(\psi)\leq L^*$ and that $\lambda G_L^0(\psi_L)$ is 
additively closed.  This time,  $L^p\subseteq G_L^0(\psi_L)$, and
therefore, if $r=1+kp$ for some integer $k\geq 0$, then for
any $c_1,\ldots ,c_r\in \lambda G_L^0(\psi_L)$, we have that
$\prod_{i=1}^rc_i\in \lambda G_L^0(\psi_L)$, and any
sum of such products is again in  $\lambda G_L^0(\psi_L)$.

Now since $F(\lambda)/F$ is separable, there is some $a_m\neq 0$ where
$1\leq m\leq n$ with $m\not\equiv 0\bmod p$ (if $n\not\equiv 0\bmod p$,
one may choose $m=n$).    Since $\gcd (p,m)=1$, there are integers
$s,k>0$ such that $sm=1+kp$.  Note that $-1=(-1)^p\in G_L^0(\psi_L)$
and that $a_m$ is a sum of $m$-fold products of
elements from 
$\{ \pm\lambda_1,\ldots,\pm\lambda_n\}\subseteq \lambda G_L^0(\psi_L)$.
Hence $a_m^s$ is a sum of $sm$-fold products of elements 
from $\lambda G_L^0(\psi_L)$, and by what was shown above using that
$sm=1+kp$, we get $0\neq a_m^s\in \lambda G_L(\psi_L)$, hence
$\phi_L\cong a_m^s\psi_L$, and Proposition \ref{sepiso}
(for $p$-forms) implies that $\phi\cong a_m^s\psi$.
\end{proof}

\section{Similarity of quadratic forms over odd degree extensions}

We have seen that totally singular quadratic forms that become similar
over a separable extension are already similar over the base field.
For general quadratic forms, this is generally not true.
For example, suppose that $F$ is not separably quadratically
closed and let $a\in F\setminus\wp(F)$ 
(where $\wp(F)=\{ c^2+c\,|\,c\in F\}$).  Let $L=F(\lambda)$ where
$\lambda=\wp^{-1}(a)$ is a root of $X^2+X+a$.  Then the quadratic form
$q=[1,a]$ is anisotropic, hence not similar to the hyberbolic plane $\HH$,
but over $L$, we have $q_L\cong\HH_L$.

We want to show, that for arbitrary quadratic forms, similarity
over an odd degree extension implies similarity over the base field.
In the proofs for bilinear forms and totally singular quadratic
forms, we have used splitting fields of polynomials.  These types
of arguments won't work in general for arbitrary quadratic forms,
since splitting fields of odd degree polynomials may have even degree
over the base field, and as we have seen, a descent of similarity
from even degree separable extensions to the base field
will generally fail.  Instead, we use transfer arguments.
But for these arguments to work, once again we have to pay
particular attention to the case of totally singular
quadratic forms.

Let us recall some basic facts about transfers of quadratic or
bilinear forms.  For more details and facts, we refer to
\cite[\S 20]{ekm}.  Let $L/F$ be a finite field extension and let
$s:L\to F$ a nonzero $F$-linear functional.  
Let $\phi$ be a bilinear resp. quadratic
form over $L$ with underlying $L$-vector space $V$.
Then we define a bilinear resp. quadratic form $s_*(\phi)$ 
over $F$ as follows.
As underlying $F$-vector space, we take again $V$ considered
as $F$-vector space and write $V_F$ to avoid confusion, in particular,
$\dim_F V_F=[L:F]\dim_L V$, and we define $s_*(\phi)=s\circ \phi$.
Note that the transfer of a nonsingular form is again nonsingular,
and that the transfer respects orthogonal sums and isometries.

We will need Frobenius reciprocity (\cite[Prop.~20.3]{ekm})
in the following situation:
If $b$ is a bilinear form over $L$ and $q$ is a quadratic form
over $F$, then 
$$s_*(b\otimes q_L)\cong s_*(b)\otimes q\quad\mbox{over $F$}.$$

We now fix the following notations.
\begin{notation}\label{nota}
Let $K=F(\lambda)$ be a simple extension of $F$.  Suppose $[K:F]=n$, 
so that $\BB=\{ 1,\lambda,\ldots,\lambda^{n-1}\}$
is an $F$-basis of $K$, and let $s:K\to F$ be the $F$-linear functional
defined by $s(1)=1$ and $s(\lambda^i)=0$ for $1\leq i\leq n-1$.
Assume furthermore that the minimal polynomial $f(X)\in F[X]$ of $\lambda$
is given by $f(X)=X^n+a_1X^{n-1}+\ldots + a_{n-1}X+a_n$. 
Note that we allow $n=1$ to avoid case distinctions.  In that case,
$f(X)=X-\lambda$, so $a_n=a_1=-\lambda=\lambda$.  Note also
that $(-1)^na_n=a_n=N_{K/F}(\lambda)$.  

Furthermore, if $K/F$ is separable, i.e., $f$ is a separable
polynomial, then $\chr(F)=2$ implies that
there exists an odd $k\in \{ 1,\ldots,n\}$
with $a:=a_k\neq 0$.  If $n$ is odd, we may choose $a=a_n$.
\end{notation}

\begin{lem}\label{s*}
With the notations in \ref{nota}, we have:
\begin{enumerate}
\item[(i)] The Gram matrix $G_1=(a_{ij})\in M_n(F)$
of $s_*(\qf{1}_b)$ with respect to the basis $\BB$
is given by $a_{ij}=s(\lambda^{i+j-2})$.
In particular, 
$q_{s_*(\qf{1}_b)}\cong\qf{1,s(\lambda^2),\ldots,s(\lambda^{2n-2})}$.

\item[(ii)] The Gram matrix $G_\lambda=(b_{ij})\in M_n(F)$
of $s_*(\qf{\lambda}_b)$ with respect to the basis $\BB$
is given by $b_{ij}=s(\lambda^{i+j-1})$.
In particular,
$q_{s_*(\qf{\lambda}_b)}\cong\qf{s(\lambda),s(\lambda^3),
\ldots,s(\lambda^{2n-1})}$.
\item[(iii)] In $WF$, the following holds:
$$\begin{array}{rcll}
s_*(\qf{1}_b) & = & \left\{\begin{array}{l}
\qf{1}_b\\ \qf{1,a}_b\end{array}\right. &
\begin{array}{l} \mbox{if $n$ is odd,}\\
\mbox{if $n$ is even.}\end{array}\\
s_*(\qf{\lambda}_b) & = & \left\{\begin{array}{l}
\qf{a}_b\\ 0\end{array}\right. &
\begin{array}{l} \mbox{if $n$ is odd,}\\
\mbox{if $n$ is even.}\end{array}
\end{array}$$
\end{enumerate}
In particular, $s_*(\pff{\lambda}_b)=\pff{a}_b\in WF$ and
$q_{s_*(\pff{\lambda}_b)}\cong\qf{1,s(\lambda),\ldots,s(\lambda^{2n-1})}$.
\end{lem}
\begin{proof} 
The statements about the Gram matrices are clear.  For the computations
regarding $s_*(\qf{1}_b)$ resp. $s_*(\qf{\lambda}_b)$ in $WF$, see
\cite[Lemmas 20.9, 20.12]{ekm} (or \cite[Ch.~VII, Theorems 2.2, 2.3]{lam},
the computation there in characteristic not $2$ can easily be adapted to
characteristic $2$ by replacing hyperbolic planes by metabolic planes).  Finally, 
the diagonalizations of $q_{s_*(\qf{1}_b)}$ resp. $q_{s_*(\qf{\lambda}_b)}$
are gotten by taking the diagonal entries of the respective
Gram matrices.
\end{proof}

\begin{cor}\label{nondegs*}
With the notations in \ref{nota}, let $q$ be a nonsingular
quadratic form over $F$. Then 
$s_*(\pff{\lambda}_b\otimes q_K)\sim \pff{a_n}_b\otimes q$.
If $n$ is odd, then $s_*(q_K)\sim q$ and
$s_*(\lambda q_K)\sim a_nq$.
\end{cor}
\begin{proof} Note that if $b,b'$ are bilinear forms over $F$
with $b\sim b'$, then $b\otimes q\sim b'\otimes q$.
Now $q_K\cong\qf{1}_b\otimes q_K$,
$\lambda q_K\cong\qf{\lambda}_b\otimes q_K$
and the result follows by Frobenius reciprocity and
Lemma \ref{s*}.
\end{proof}

\begin{lem}\label{tss*}
With the notations in \ref{nota}, assume furthermore that $K/F$
is separable. Let $\phi$ and
$\psi$ be totally singular quadratic forms over $F$ and 
suppose that $\phi_K\cong \lambda\psi_K$.
Let $m$ be a nonnegative integer.  Then
$$s(\lambda^m)\in\left\{
\begin{array}{ll}
G_F^0(\phi)=G_F^0(\psi) & \mbox{if $m\equiv 0\bmod 2$,}\\
G_F^0(\phi,\psi)=aG_F^0(\psi) & \mbox{if $m\equiv 1\bmod 2$}
\end{array}
\right.$$  
(where $a\in F^*$ is chosen as in \ref{nota}).
In particular,
$$\phi\sim s_*(\phi_K)\cong s_*(\lambda\psi_K)\sim a\psi.$$
\end{lem}
\begin{proof}
We already know that the similarity of $\phi$ and $\psi$ 
over $K$ implies that over $F$, in fact, we have 
$\phi\cong a\psi$ by Lemma \ref{lambda}, and therefore also clearly
$G_F^0(\phi)=G_F^0(\psi)$ and $a\in G_F^0(\phi,\psi)$, i.e.,
$G_F^0(\phi,\psi)=aG_F^0(\psi)$.

We prove the statement about $s(\lambda^m)$ by induction.
By definition of $s$, we have  $s(\lambda^0)=s(1)=1\in G_F^0(\psi)$,
$s(\lambda^m)=0\in  G_F^0(\psi)\cap G_F^0(\phi,\psi)$  for
$1\leq m\leq n-1$, and by Lemma \ref{lambda}
$$s(\lambda^n)=s\bigl(\sum_{i=1}^na_i\lambda^{n-i}\bigr)=
\sum_{i=1}^na_is(\lambda^{n-i})=a_n\in
\left\{\begin{array}{ll}
G_F^0(\phi,\psi) & \mbox{if $n$ is odd,}\\
G_F^0(\psi) &  \mbox{if $n$ is even.}\end{array}\right.$$
So the statement is true for $m\leq n$.  
Suppose the statement is true for all $k<m$ for some $m>n$.
Then 
\begin{equation}\label{slambda}
s(\lambda^m)=s\bigl(\sum_{i=1}^na_i\lambda^{m-i}\bigr)=
\sum_{i=1}^na_is(\lambda^{m-i}).
\end{equation}
Suppose first that $m$ is even.  If $i$ is even, then $m-i$ is
also even, hence, $a_i\in G_F^0(\psi)$ by Lemma \ref{lambda}
and $s(\lambda^{m-i})\in G_F^0(\psi)$ by induction, 
thus $a_is(\lambda^{m-i})\in G_F^0(\psi)$.  
If $i$ and therefore also $m-i$ are odd,  we have
$a_i\in G_F^0(\phi,\psi)$ by Lemma \ref{lambda}
and $s(\lambda^{m-i})\in G_F^0(\phi,\psi)$ by induction, 
and again we get $a_is(\lambda^{m-i})\in G_F^0(\psi)$.
We conclude that each summand on the right in Eq.~\ref{slambda}
is in $G_F^0(\psi)$, and since $G_F^0(\psi)$ is additively
closed, we get that $s(\lambda^m)\in G_F^0(\psi)$.

Using a similar reasoning, if $m$ is odd and if
$i$ is even (resp. odd) then $m-i$
is odd (resp. even) and thus $a_i\in G_F^0(\psi)$ and
$s(\lambda^{m-i})\in G_F^0(\phi,\psi)$
(resp. $a_i\in G_F^0(\phi,\psi)$ and $s(\lambda^{m-i})\in G_F^0(\psi)$)
and it follows that in Eq.~\ref{slambda}, we have 
$a_is(\lambda^{m-i})\in G_F^0(\phi,\psi)$ for each $1\leq i\leq n$
and hence also $s(\lambda^m)\in G_F^0(\phi,\psi)$.

Now 
$$s_*(\phi_K)\cong s_*(\qf{1}_b)\otimes \phi\cong
q_{s_*(\qf{1}_b)}\otimes\phi\cong
\qf{1,s(\lambda^2),\ldots,s(\lambda^{2n-2})}\otimes \phi.$$
Let $m$ be even, $2\leq m\leq 2n-2$.  Then either $s(\lambda^m)=0$
and $s(\lambda^m)\phi\sim 0$, or $s(\lambda^m)\in G_F(\phi)=G_F(\psi)$
and thus $\phi\cong  s(\lambda^m)\phi$.  Note also that
$\phi\perp\phi\sim\phi$ and therefore
$$\qf{1,s(\lambda^2),\ldots,s(\lambda^{2n-2})}\otimes \phi\sim\phi$$
Similarly, we get 
$$s_*(\phi_K)\cong s_*(\qf{\lambda}_b)\otimes \psi\cong 
q_{s_*(\qf{\lambda}_b)}\otimes\psi\cong
\qf{s(\lambda), s(\lambda^3),\ldots,
s(\lambda^{2n-1})}\otimes \psi.$$
Let $m$ be odd, $1\leq m\leq 2n-1$. Then either $s(\lambda^m)=0$ and
$s(\lambda^m)\psi\sim 0$, or $s(\lambda^m)\in G_F(\phi,\psi)$ and thus
$a\psi\cong  s(\lambda^m)\psi$. 
Note that since $\chr(F)=2$ and $F(\lambda)/F$ is separable,
we must have $F(\lambda^2)=F(\lambda)$, i.e., $[F(\lambda^2):F]=n$,
so in particular,
$\{ 1,\lambda^2,\ldots,(\lambda^2)^{n-1}\}$ and therefore also
$\{ \lambda,\lambda^3,\ldots,\lambda^{2n-1}\}$ are  $F$-bases
of $F(\lambda)$.  Since $s$ is a nonzero
$F$-linear functional, there must be at least one such odd $m$
with $s(\lambda^m)\neq 0$, and
similarly as above, we now get
$$\qf{s(\lambda), s(\lambda^3),\ldots,s(\lambda^{2n-1})}\otimes
\psi\sim a\psi.$$
\end{proof}

\begin{theorem}\label{qfoddsim}
Let $\phi$ and $\psi$ be quadratic forms over $F$ of the same
dimension and let $L/F$ be an odd degree field extension.
Then $\phi$ is similar to $\psi$ over $F$ if and only 
if $\phi$ is similar to $\psi$ over $L$.  More precisely,
if $\lambda\in L^*$ with $\phi_L\cong\lambda\psi_L$, then,
with the same notations as in \ref{nota}, $\phi\cong a_n\psi$
(over $F$).
\end{theorem}

\begin{proof}
Write $\phi\cong\phi_r\perp\phi_s$ with $\phi_r$ nonsingular
and $\phi_s$ totally singular, and analogously $\psi\cong\psi_r\perp\psi_s$.
The dimensions of the nonsingular resp. totally singular part do not
change under field extensions and remain unchanged under isometries
or similarities.  Hence, we may in fact assume
$\dim\phi_r=\dim\psi_r$ and $\dim\phi_s=\dim\psi_s$.

The ``only if''-part in the theorem being trivial, assume
conversely that $\phi_L$ is similar to $\psi_L$ and let
$\lambda\in L^*$ be such that $\phi_L\cong\lambda \psi_L$.
Let $K=F(\lambda)$.  Then $L/K$ is an odd degree extension
and by Proposition \ref{qfoddiso}, we have $\phi_K\cong\lambda\psi_K$.
Hence, with the notations in \ref{nota} and
by Corollary \ref{nondegs*} and Lemma \ref{tss*} (where we may choose
$a=a_n$), we obtain
$$\begin{array}{rcl}
\phi & \cong & \phi_r\perp\phi_s\\
 & \sim & s_*((\phi_r)_K)\perp s_*((\phi_s)_K)\\
 & \cong & s_*(\phi_K)\\
 & \cong & s_*(\lambda\psi_K)\\
 & \cong & s_*(\lambda(\psi_r)_K)\perp s_*(\lambda(\psi_s)_K)\\
 & \sim & a_n\psi_r\perp a_n\psi_s\\
 & \cong & a_n\psi,\end{array}$$
so we have $\phi\sim a_n\psi$ and comparing dimensions of the 
nonsingular and totally singular parts, we get $\phi\cong a_n\psi$.
\end{proof}

\section{Scharlau's norm principle}
To treat the norm principle in its fullest generality,
we again need to invoke properties of totally singular
quadratic forms under finite extensions.  We need the
following corollary to Proposition \ref{tsv}.
\begin{lem}\label{tsnp-lem}
Let $q=\qf{a_1,\ldots,a_n}$ be a nonzero totally singular
quadratic form over $F$, let $L/F$ be a finite extension,
and let $\lambda\in L^*$ and $K=F(\lambda)$.
If $\lambda\in G_L(q_L)$, then
$K/F$ is separable
or $L/K$ is inseparable.
\end{lem}

\begin{proof}  If $\lambda\in F$ then $K=F$ is trivially separable
over $F$.  If $\lambda\in L\setminus F$, 
the lemma follows from Proposition \ref{tsv}(iii) as $q_L\cong\lambda q_L$ if
and only if $\lambda a_i\in D_L(q_L)$ for all $1\leq i\leq n$
and because at least one $a_i$ is nonzero.
\end{proof}

If $L/F$ is a finite extension, we denote by $N_{L/F}:L\to F$ 
the usual norm map. We recall well known properties.
If $L/K/F$ is a tower of finite extensions,
then $N_{L/F}=N_{K/F}\circ N_{L/K}$, and if $x\in K$, then
$N_{L/F}(x)=N_{K/F}(x)^{[L:K]}$.

As a consequence of Lemma \ref{tsnp-lem},
we obtain Scharlau's norm principle for
totally singular quadratic forms.

\begin{cor}\label{tsnp-cor}
Let $q$ be a totally singular quadratic form over $F$ and
let $L/F$ be a finite field extension.  Then
$N_{L/F}(G_L(q_L))\subseteq G_F(q)$.
\end{cor}

\begin{proof}  If $q$ is a zero form, then
$G_L(q_L)=L^*$ and clearly 
$N_{L/F}(G_L(q_L))=N_{L/F}(L^*)\subseteq F^*=G_F(q)$.

So assume $q$ is not a zero form and let $\lambda\in G_L(q_L)$
and put $K=F(\lambda)$ and let $a=N_{K/F}(\lambda)\in F^*$.
With the notations as in \ref{nota}, we have $a=a_n$.
By Lemma \ref{tsnp-lem}, we have $K/F$ separable or
$L/K$ inseparable, and the latter implies that $[L:K]$ is even.  
If $K/F$ is separable then
$a\in G_F(q)=G_F(q,q)$  by Lemma \ref{lambda}.
If $[L:K]$ is even, then $a^{[L:K]}\in F^{*2}\subseteq G_F(q)$.
In any case,
$a^{[L:K]}=N_{L/F}(\lambda)\in G_F(q)$.
\end{proof}

We now get the full version of Scharlau's norm principle
for arbitrary quadratic or bilinear forms.

\begin{theorem}\label{snp}
Let $\phi$ be a quadratic or bilinear form over $F$
and let $L/F$ be a finite field extension.
Then $N_{L/F}(G_L(\phi_L))\subseteq G_F(\phi)$.
\end{theorem}

\begin{proof}
Let $K=F(\lambda)$ and put
$a=N_{K/F}(\lambda)$. If $[L:K]$ is even, then 
\begin{equation}\label{even}
N_{L/F}(\lambda)=a^{[L:K]}\in F^{*2}\subseteq G_F(\phi)
\end{equation}
and we are done (no matter whether $\phi$ is a quadratic
or bilinear form).  So for the remainder of the proof,
let us assume that $[L:K]$ is odd.

\medskip

\noindent\emph{The bilinear case}. 
Suppose that $\phi$ is a bilinear form.
If $\phi$ is hyperbolic, i.e., $q_\phi$ is a zero form,
then clearly $G_L(\phi_L)=L^*$ and
$N_{L/F}(G_L(\phi_L))=N_{L/F}(L^*)\subseteq F^*=G_F(\phi)$.

So assume that $\phi$ is not hyperbolic, i.e., 
$q_\phi$ is not a zero form.
Then $\phi_L\cong\lambda\phi_L$ implies in particular
$(q_\phi)_L\cong \lambda (q_\phi)_L$.   
Now $[L:K]$ is odd and hence $L/K$ is separable, so on the one hand,
$\lambda\in G_L((q_\phi)_L)$ implies 
that $K/F$ is separable by Lemma \ref{tsnp-lem},
and on the other hand $\phi_L\cong\lambda\phi_L$
implies $\phi_K\cong\lambda\phi_K$ by Proposition \ref{sepiso}.
Using Lemma \ref{lambda} and the notations there
(with $K$ instead of $L$ and $a=a_n$), we see that
$a\in G_F(\phi)=G_F(\phi,\phi)$.  Hence, again,
$$N_{L/F}(\lambda)=a^{[L:K]}\in G_F(\phi).$$

\medskip

\noindent\emph{The quadratic case.}
Now assume that $\phi$ is a quadratic form and write
$\phi\cong\phi_r\perp\phi_s$ with $\phi_r$ nonsingular
and $\phi_s$ totally singular. 
Recall that we assume that $[L:K]$ is odd, so
$\phi_L\cong \lambda\phi_L$ implies 
$\phi_K\cong\lambda\phi_K$ by Proposition \ref{qfoddiso},
and by Proposition \ref{qfwittiso}, we get 
$(\phi_s)_K\cong\lambda (\phi_s)_K$ and
\begin{equation}\label{snp-eq0}
(\phi_s)_K \sim  (\phi_r)_K\perp\lambda (\phi_r)_K\perp (\phi_s)_K
\cong\pff{\lambda}_b\otimes(\phi_r)_K\perp (\phi_s)_K.
\end{equation}
We use the notations in \ref{nota} and we obtain
that 
\begin{equation}\label{snp-eq1}
a\phi_s\cong\phi_s\sim s_*((\phi_s)_K).
\end{equation}
In fact, this is obvious if $\phi_s$
is the zero form. 
If $\phi_s$ is not the zero form, then 
$(\phi_s)_K\cong \lambda (\phi_s)_K$
forces $K/F$ to be separable 
by Lemma \ref{tsnp-lem}, and hence we can apply 
Lemmas \ref{lambda} and \ref{tss*}
to get Eq.~\ref{snp-eq1}.

Invoking Corollary \ref{nondegs*} and 
Eqs.~\ref{snp-eq0}, \ref{snp-eq1}, we get
$$\begin{array}{rcl}
a\phi_s\cong\phi_s & \sim &  s_*\bigl((\phi_s)_K\bigr)\\
 & \sim & s_*\bigl(\pff{\lambda}_b\otimes (\phi_r)_K\perp (\phi_s)_K\bigr)\\
 & \sim & s_*\bigl(\pff{\lambda}_b\otimes (\phi_r)_K\bigr)
\perp s_*\bigl((\phi_s)_K\bigr)\\
 & \sim & \pff{a}_b\otimes\phi_r\perp \phi_s\\
 & \sim & \phi_r\perp a\phi_r\perp\phi_s.
\end{array}$$
Thus, we can apply Proposition \ref{qfwittiso} to conclude that
$\phi\cong\phi_r\perp\phi_s\cong a\phi_r\perp a\phi_s\cong a\phi$.
Hence $a\in G_F(\phi)$, and again we get
$N_{L/F}(\lambda)=a^{[L:K]}\in G_F(\phi)$.
\end{proof}

\end{document}